\documentclass[12pt,a4paper,reqno]{amsart}
\usepackage[mathcal]{eucal}
\usepackage{amssymb}
\usepackage[nice]{nicefrac}
\usepackage{hyperref}
\usepackage[normalem]{ulem}
\usepackage{xcolor}

\usepackage{amsthm}
\usepackage{mathtools}
\usepackage{amsfonts}
\usepackage{comment}
\usepackage{faktor}

\theoremstyle{plain}
\newtheorem{theorem}{Theorem}[section]
\newtheorem{lemma}[theorem]{Lemma}

\theoremstyle{definition}
\newtheorem{definition}[theorem]{Definition}

\numberwithin{equation}{section}

\makeatletter
\newcommand{\AlignFootnote}[1]{%
    \ifmeasuring@
    \else
        \footnote{#1}%
    \fi
}
\makeatother

\title{On soluble groups in which commutators have prime power order}

\author[M. Figueiredo]{Mateus Figueiredo} 
\address{Mateus Figueiredo: Department of Mathematics, University of Brasilia, Brasilia DF, Brazil}
\email{mt-figueiredo@hotmail.com}

\author[P. Shumyatsky]{Pavel Shumyatsky} 
\address{Pavel Shumyatsky: Department of Mathematics, University of Brasilia, Brasilia DF, Brazil}
\email{pavel@unb.br}

\thanks{The work of the second author was supported by FAPDF and CNPq}


\begin{document}

\maketitle

\begin{abstract}
The article deals with finite groups  in which commutators have prime power order (CPPO-groups). We show that if $G$ is a soluble CPPO-group, then the order of the commutator subgroup $G'$ is divisible by at most two primes. This improves an earlier result saying that the order of the commutator subgroup $G'$ is divisible by at most three primes.
\end{abstract}

\section{Introduction}

Finite groups in which every element has prime power order (EPPO-groups for short) are nowadays fairly well understood. A rather detailed description of such groups (including infinite locally finite EPPO-groups) can be found in \cite{delgado}. In particular, a finite soluble EPPO-group has Fitting height at most 3 and its order is divisible by at most two primes \cite{higman}. 

The recent article \cite{MS1} deals with finite groups  in which commutators have prime power order (CPPO-groups). Roughly, the results obtained in \cite{MS1} show that if $G$ is a CPPO-group, then the structure of the commutator subgroup $G'$ is similar to that of an EPPO-group. For instance, the Fitting height of a soluble CPPO-group $G$ is at most 3 and the order of the commutator subgroup $G'$ is divisible by at most three primes.

The goal of this short note is to improve the above result and establish the following theorem.

\begin{theorem}\label{main}
Let $G$ be a soluble CPPO-group. Then the order of the commutator subgroup $G'$ is divisible by at most two primes.
\end{theorem} 
In the next section we collect auxiliary results needed for the proof of Theorem \ref{main}. The proof is given in Section 3.

\section{Preliminaries}

If $A$ is a group of automorphisms of a group $G$, the subgroup generated by all elements of the form $g^{-1}g^\alpha$ with $g\in G$ and $\alpha\in A$ is denoted by $[G,A]$. It is well known that the subgroup $[G,A]$ is an $A$-invariant normal subgroup in $G$. We write $C_G(A)$ for the centralizer  of $A$ in $G$. If $G$ and $A$ are finite and $(|G|,|A|)=1$, we say that $A$ is a group of coprime automorphisms of $G$. Throughout, $\pi(G)$ denotes the set of prime divisors of the order of $G$. We write $\Phi(G)$ to denote the Frattini subgroup of $G$. The following lemma is a collection of well-known facts on coprime actions (see for example \cite[Ch. 5 and Ch. 6]{Gorenstein}). In the sequel the lemma will be used without explicit references. 

\begin{lemma}\label{cc}
Let  $A$ be a group of coprime automorphisms of a finite group $G$. Then
\begin{enumerate}
\item[(i)] $G=[G,A]C_{G}(A)$. If $G$ is abelian, then $G=[G,A]\oplus C_{G}(A)$.
\item[(ii)] $[G,A,A]=[G,A]$. 
\item[(iii)] $C_{G/N}(A)=NC_G(A)/N$ for any $A$-invariant normal subgroup $N$ of $G$.
\item[(iv)] If $[G/\Phi(G),A]=1$, then $[G,A]=1$
\end{enumerate}
\end{lemma}

We will require the concept of towers as introduced by Turull (see \cite{Turull}).

\begin{definition}\label{DefTurulltower}
	Let $G$ be a group. A sequence $(P_{i})_{i=1,\ldots,h}$ of subgroups of $G$ is said to be a tower of height $h$ if the following are satisfied:
	\begin{enumerate}
		\item $P_{i}$ is a $p_i$-group for all $i=1,\ldots,h$.
		\item $P_{i}$ normalizes $P_j$ for all $i<j$.
		\item Put $\overline{P_h}=P_{h}$ and $\overline{P_{i}}={P_{i}}/{C_{P_{i}}(\overline{P_{i+1}})}, \  i=1,\ldots,h-1.$ 
		Then $\overline{P_i}$ is nontrivial for all $i$.
		\item $p_i\neq p_{i+1}$ for all $i=1,\ldots,h-1$.
	\end{enumerate}
\end{definition}

The next lemma is due to Turull \cite[Lemma 1.9]{Turull}.
\begin{lemma}\label{heightequaltower} The Fitting height of a finite soluble group $G$ is the maximum of heights of towers of $G$.
\end{lemma}

In the sequel the following lemma from \cite[Lemma 3.1]{MS1} will be useful.
\begin{lemma}\label{abeliantower}
	Let $G$ be a finite group containing a tower $(P_1,P_2,P_3)$ of abelian subgroups such that $P_1$ is cyclic, $P_2$ is noncyclic and $P_2=[P_2,P_1]$. Then $G$ has a commutator whose order is not a prime power. 
\end{lemma}

The next lemma provides useful tool for handling  CPPO-groups. It is taken from \cite[Lemma 3.2]{MS1}.
\begin{lemma}\label{casoextra}
	Let $G$ be a CPPO-group containing a tower $(P_1,P_2,P_3)$ with the following properties:
	\begin{enumerate}
		\item $P_1$ is cyclic;
		\item $P_2$ is extraspecial and $C_{P_2}(P_1)=\Phi(P_2)$;
		\item $P_3$ is abelian and $P_3=[P_3,\Phi(P_2)]$.
	\end{enumerate}
	Then $p_2=2$ and $P_2$ is isomorphic to the quaternion group $Q_8$. 
\end{lemma}

\begin{lemma}\label{MSmain}
Let $G$ be a soluble CPPO-group. Then $h(G)\leqslant3$ and $|\pi(G')|\leqslant h(G)$. Furthermore, set $K_1=G$ and $K_{i+1}=\gamma_\infty(K_i)$ for $i=1,2,\dots$. Then the orders of $G'/K_2$, $K_2/K_3$, and $K_3$ are prime powers.
\end{lemma}
\begin{proof} It is shown in \cite{MS1} that $h(G)\leqslant3$ and $|\pi(G')|\leqslant 3$. The claim about the orders of $G'/K_2$, $K_2/K_3$, and $K_3$ was essentially established in the proof of Proposition 3.8 in \cite{MS1}. The fact that $|\pi(G')|\leqslant h(G)$ is now immediate.
\end{proof}

Let $G$ be a finite soluble group. Recall that the group $G$ has a Sylow basis. This is a family of pairwise permutable Sylow $p_i$-subgroups
$P_i$ of $G$, exactly one for each prime, and any two Sylow bases are
conjugate. The system normalizer (also known as the basis normalizer) of such a Sylow basis in $G$ is $T=\cap_iN_G(P_i)$. 

The reader can consult \cite[pp. 235–240]{Doerk} or \cite[pp. 262--265]{rob}
for general information on system normalizers. In particular, we will use the facts
that the system normalizers are conjugate, nilpotent, and $G = T\gamma_\infty(G)$, where
$\gamma_\infty(G)$ denotes the intersection of the lower central series of $G$. Moreover, if $\gamma_\infty(G)$ is abelian then $T\cap\gamma_\infty(G)=1$.
Furthermore, system normalizers are preserved by epimorphisms, that is, the image of $T$ in
any quotient $G/N$ is a basis normalizer of $G/N$. 

\section{Proof of Theorem \ref{main}}

\begin{proof} 
	Suppose $G$ is a counterexample with minimal possible order. In view of Lemma \ref{MSmain} the Fitting height of $G$ is precisely $3$. Set $K_1=G$, $K_2=\gamma_\infty(K_1)$, and $K_3=\gamma_\infty(K_2)$.
	Choose a Sylow basis $\{P_j\}$ of $G$ and let $T_i$ be the system normalizer of the induced Sylow basis $\{P_j\cap K_i\}$ of $K_i$. As $h(G)=3$, it follows that $G=T_1T_2T_3$. Observe that $T_2T_3=\gamma_\infty(G)$ and $T_3=\gamma_\infty(\gamma_\infty(G))$.
	
Since $G$ is a counterexample of minimal order, $T_3$ is a minimal normal subgroup of $G$. Hence, $T_3$ is an elementary abelian $p$-group for some prime number $p$. It follows that $T_2\cap T_3=1$. On the other hand, $K_2/T_3$ is a $q$-group for some prime number $q\neq p$, as shows Lemma \ref{MSmain}. We conclude that $T_2$ is a Sylow $q$-subgroup of $\gamma_\infty(G)$. Moreover, $T_3=[T_3,T_2]$. 
	
By Lemma \ref{MSmain} the derived subgroup of $G/K_2$ is an $r$-group for some prime number $r$. Consequently the Sylow $r$-subgroup $R$ of $T_1$ is not abelian. We claim that $[T_2,R]\neq1$. Indeed, remark that $[T_2,T_1]\neq1$ and suppose $[T_2,R]=1$. There are elements $a_1,a_2\in R$, an $r'$-element $x\in T_1$ and $b\in T_2$ such that $[a_1,a_2]\neq 1$ and $[x,b]\neq 1$. But then the commutator $[a_1x,a_2b]=[a_1,a_2][x,b]$ has order divisible by both $q$ and $r$, which yields a contradiction. Therefore $[T_2,R]\neq1$ and by minimality of $G$ we obtain that $G=RT_2T_3$.
	
Let us show that $R$ is 2-generated and has nilpotency class 2. Let $n$ be the largest non-negative integer satisfying the property that $[T_2,Z_n(R)]=1$, where $Z_n(R)$ stands for the $n$th term of the upper central series of $R$. Observe that $Z_n(R)<R$. Choose $a_1\in Z_{n+1}(R)$ such that $[T_2,a_1]\neq1$ and $b\in T_2$ such that $[a_1,b]\neq1$. If $n\geq1$, then there is an element $a_2\in R$ such that $[a_1,a_2]\neq1$. The equality $[a_1,a_2b]=[a_1,b][a_1,a_2]$ shows that $G$ contains a commutator whose order is divisible by both $q$ and $r$. This contradiction shows that $n=0$. 
	
Recall that $R$ is not abelian. Choose $a_1\in Z_2(R)\setminus Z(R)$ and $a_2\in R$ such that $a=[a_1,a_2]\neq1$. Of course, $a\in Z(R)$. We will now show that
	$$C_{K_2}(a)=C_{K_2}(\langle a_1,a_2\rangle).$$
Let $x\in C_{K_2}(a)$. We have $[xa_1,a_2]=[x,a_2]^{a_1}a$. Since this has prime power order, it follows that $[x,a_2]=1$ and $x\in C_{K_2}(a_2)$. Similarly one obtains that $x\in C_{K_2}(a_1)$. So indeed $C_{K_2}(a)=C_{K_2}(\langle a_1,a_2\rangle)$. 

If $[T_2,a]=1$, then also $[T_2,a_1]=[T_2,a_2]=1$. Taking $z\in Z(R)$ and $b\in T_2$ such that $[z,b]\neq 1$, observe that the order of the commutator 
	$$[a_1z,a_2b]=[a_1,a_2b]^z[z,a_2b]=[a_1,a_2][z,b]$$ is not a prime power. This is a contradiction. 
Therefore $[T_2,a]\neq1$. We see that the subgroup $\langle a_1,a_2\rangle T_2T_3$ is also a counterexample to the theorem. In view of minimality of $G$ conclude that $R=\langle a_1,a_2\rangle$ has nilpotency class 2. 
	
Remark that $G/C_{T_2}(T_3)$ is also a counterexample to the theorem so because of minimality of $G$ we deduce that $C_{T_2}(T_3)=1$. Therefore $T_2$ does not contain nontrivial proper subgroups $K$ such that $K=[K,R]=[K,a]$. In particular, $T_2=[T_2,a]$ and $T_2$ is its Thompson's critical subgroup (see \cite{Gorenstein}, pp. 185-186). Hence the nilpotency class of $T_2$ is at most 2.
	
It is not difficult to check that $T_2$ is not abelian. Indeed, suppose the opposite and note that in this case $T_2$ is elementary abelian. If $T_2$ is a noncyclic group, the tower $(\langle a\rangle,T_2,T_3)$ satisfies the conditions of Lemma \ref{abeliantower} and so $G$ contains a commutator whose order is not a prime power, a contradiction. Thus, $T_2$ is a cyclic group of prime order. Recall that the automorphism group of any cyclic group is abelian. Hence, $Aut(T_2)$ is abelian and so $R/C_{R}(T_2)$ is abelian. But then we obtain that $[T_2,a]=1$, a contradiction. This proves that $T_2$ is not an abelian group. 
	
As $\Phi(T_2)=C_{T_2}(a)$, we have $Z(T_2)=T_2'=\Phi(T_2)$.
We claim that $\Phi(T_2)$ is a cyclic group. Since $RT_2$ acts irreducibly on $T_3$, the centre of $RT_2/C_{RT_2}(T_3)$ is cyclic. Taking into account that $\Phi(T_2)\leq Z(RT_2)$ and using again that $C_{T_2}(T_3)=1$ we get that $Z(T_2)$ is a cyclic group, as claimed. 
	
We will now separately consider the two cases where the prime $q$ is odd or $q=2$. 
Suppose that $q$ is an odd prime. In this case $T_2$ has exponent $q$, that is, $T_2=\Omega_1(T_2)$ (cf. \cite[pp. 184]{Gorenstein}). Therefore, $T_2$ is an extraspecial $q$-group. Using that $T_3=[T_3,\Phi(T_2)]$ and applying Lemma \ref{casoextra} to the tower $(\langle a\rangle, T_2,T_3)$ we get a contradiction.
	
So assume that $q=2$. Suppose $R$ centralizes every involution of $T_2$. Then all of them are contained in $\Phi(T_2)$. Since $\Phi(T_2)$ is cyclic, it follows that $T_2$ contains only one involution. Therefore $T_2$ is a (generalised) quaternion group. Since $Z(T_2)=T_2'$ we conclude that $T_2\cong Q_8$. As $Aut(Q_8)\cong S_4$, we get that $r=3$ and $R/C_{R}(T_2)$ has order 3. It follows that $[T_2,a]=1$, a contradiction. 
	
Therefore $R$ does not centralize some involution $x\in T_2$. Let $Q=\langle x^g;~g\in R\rangle$. So $Q$ is a subgroup of $T_2$ generated by involutions. As $1\neq[R,Q]\leq Q$, by minimality $Q=T_2$. Let $\{x_1,\ldots,x_k\}$ be the orbit of $x$ under the action of $R$. We have $T_2=\langle x_1,\ldots,x_k\rangle$ and  $T_2'=\langle [x_i,x_j],~i,j=1,\ldots,k\rangle$. Any commutator $[x_i,x_j]$ is an involution and consequently $T_2'$ is an elementary abelian group. Because $T_2'$ is cyclic, it follows that $T_2$ is extraspecial. Using again Lemma \ref{casoextra} we conclude that $T_2\cong Q_8$. However, this contradicts the fact that $T_2$ is generated by involutions. This is the final contradiction.
\end{proof}


\begin{thebibliography}{b}
\bibitem{delgado} A. Delgado, Y.-F. Wu, On locally finite groups in which every element has prime power order, Illinois J. Math., \textbf{46}  (2002), 885--891.
\bibitem{Doerk} K. Doerk and T. O. Hawkes, \textit{Finite Soluble Groups}, De Gruyter, New York, 1992.
\bibitem{MS1} M. Figueiredo and P. Shumyatsky, Finite groups in which every commutator has prime power order, J. Algebra, \textbf{658} (2024), 779--797.
\bibitem{Gorenstein} D. Gorenstein, \textit{Finite Groups}, Chelsea Publishing Company, New York, 1980.
\bibitem{higman} G. Higman, Finite Groups in Which Every Element Has Prime Power Order, J. Lond. Math. Soc., \textbf{s1-32}  (1957), 335--342.
\bibitem{rob} D. J. S. Robinson, \textit{A course in the theory of groups}, Springer, New York, 1996.
\bibitem{Turull} A. Turull,  Fitting height of groups and of fixed points, J. Algebra, \textbf{86}  (1984), 555--566.


\end{thebibliography}
\end{document}